\documentclass[12pt]{amsart}

\usepackage{graphicx}
\usepackage{amsfonts}
\usepackage{epsf}
\usepackage{amssymb}
\usepackage{amsmath}
\usepackage{amscd}
\usepackage{hyperref}

\newtheorem{theorem}{Theorem}[section]
\newtheorem{proposition}[theorem]{Proposition}
\newtheorem{lemma}[theorem]{Lemma}

\newtheorem{remark}[theorem]{Remark} 
\newtheorem{example}[theorem]{Example}

\newcommand{\eps}{\epsilon}

\input xy
\xyoption{all}

\begin{document}

\title{The Goeritz matrix and signature of a two bridge knot}
\begin{abstract}
According to a formula by Gordon and Litherland \cite{GordonLitherland}, the signature $\sigma (K)$ of a knot $K$ 
can be computed as $\sigma (K) = \sigma (G) - \mu$ where $G$ is the Goeritz matrix of a projection $D$ of $K$ while $\mu$ is a suitable \lq\lq correction term\rq\rq, read off from the same projection $D$. In this article, we consider the family of two bridge knots and compute the signature of their Goeritz matrices. In many cases we also compute the correction term $\mu$. More to the point, we show that every two bridge knot $K$ has distinguished projections for which $\mu=0$, obtaining $\sigma (K) = \sigma (G)$ for that projection. We provide an algorithm for finding such distinguished projections. 

This article is the result of an REU study conducted by the first author under the direction of the second. \end{abstract}
\author{Michael Gallaspy}
\author{Stanislav Jabuka}
\email{gallasp2@unr.nevada.edu,  jabuka@unr.edu}
\address{Department of Mathematics and Statistics, University of Nevada, Reno NV 89557, USA.}
\thanks{The second author was partially supported by the NSF grant DMS-0709625, the first author was supported by an REU supplement to this grant.}
\maketitle
\section{Introduction} \label{section-introduction}
\subsection{Statement of results} 
Given a collection of nonzero integers $c_1,...,c_n\in \mathbb Z$, the associated two bridge knot/link $K_{[c_1,c_2,...,c_n]}$ is the isotopy class of the knot diagram $D_{[c_1,...,c_n]}$ as in Figure \ref{pic1}. To the ordered collection $(c_1,...,c_n)$ we associate a rational number $p/q$ by means of its continued fraction expansion, i.e. we define  
$$[c_1,...,c_n] := c_1-\cfrac{1}{c_2-\cfrac{1}{ \ddots - \cfrac{1}{c_{n-1}-\cfrac{1}{c_n}}}}   $$ 
and set $p/q= [c_1,...c_n]$. It is a remarkable theorem of J. Conway \cite{Conway} that if two continued fractions $[c_1,...,c_n]$ and $[d_1,...,d_m]$ yield the same rational number $p/q$, then the two knots/links $K_{[c_1,...,c_n]}$ and $K_{[d_1,...,d_m]}$ are isotopic. This justifies the notation $K_{p/q}$ instead of $K_{[c_1,...,c_n]}$ which we shall employ when convenient. 
\begin{figure}[htb!] 
\centering
\includegraphics[width=14cm]{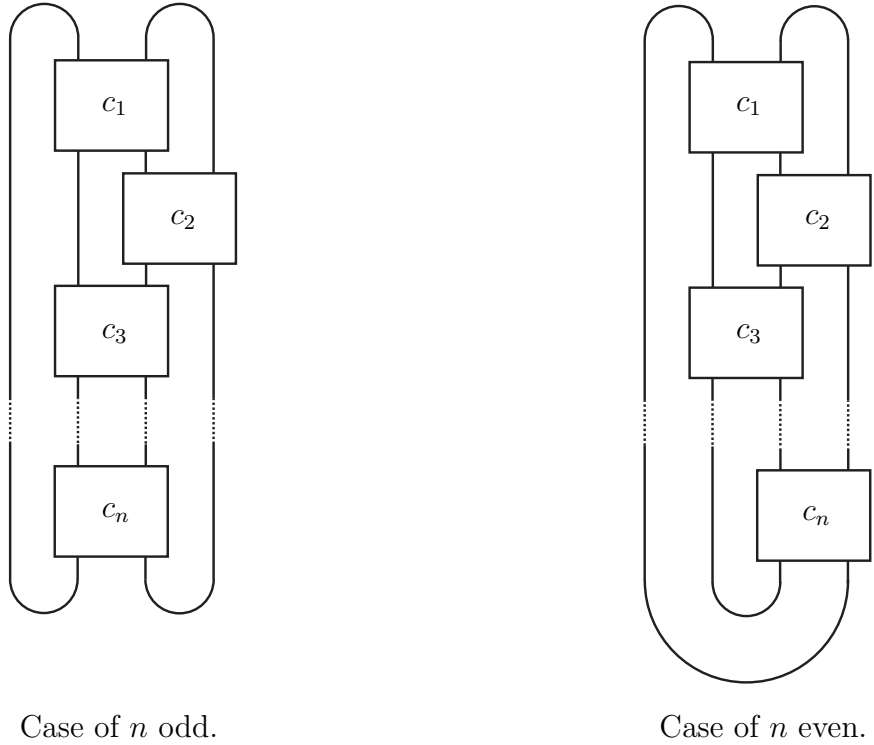}
\put(-331,222){$c_1$}
\put(-91,222){$c_1$}
\put(-305,179){$c_2$}
\put(-65,179){$c_2$}
\put(-331,136){$c_3$}
\put(-91,136){$c_3$}
\put(-331,68){$c_n$}
\put(-66,67){$c_n$}
\put(-362,-15){Case of $n$ odd.}
\put(-120,-15){Case of $n$ even.}
\caption{The two bridge knot/link associated to the integers $c_1,...,c_n\in \mathbb Z$. The meaning of each box containing an integer is as in Figure \ref{pic1b}. Throughout the article we shall refer to this projection of the knot/link $K_{[c_1,...,c_n]}$ as $D_{[c_1,...,c_n]}$. }  \label{pic1}
\end{figure}
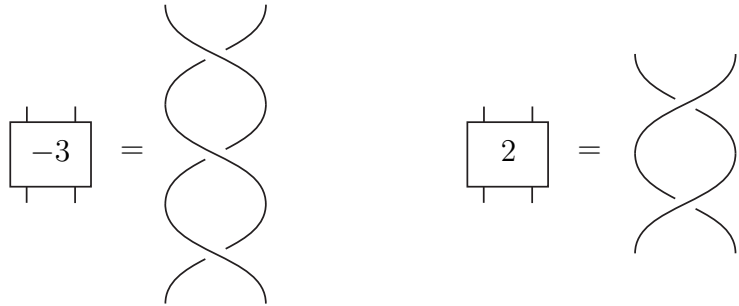
\begin{figure}[htb!] 
\centering
\includegraphics[width=10cm]{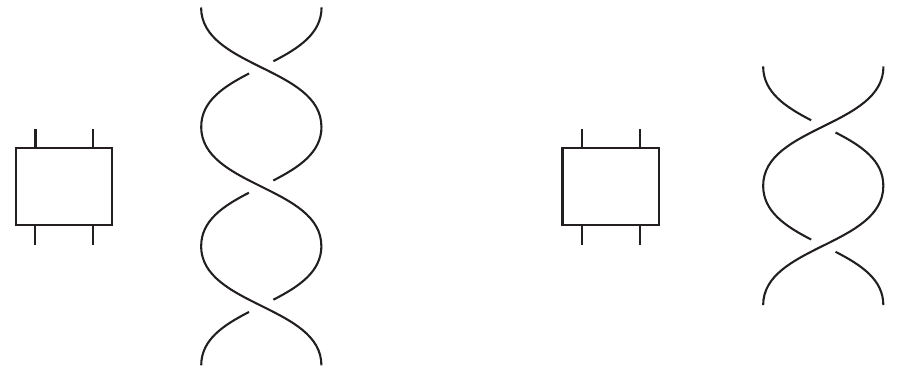}
\put(-272,57){$-3$}
\put(-238,58){$=$}
\put(-94,57){$2$}
\put(-65,58){$=$}
\caption{Each of the boxes from Figure \ref{pic1},  containing a integer $c$ and two incoming/outgoing strings, represents a pair of parallel strands with $|c|$ half-twists. Our convention is that $c>0$ corresponds to right-handed and $c<0$ to left-handed half-twists. }  \label{pic1b}
\end{figure}

Before stating our main results, we pause to define the notion of a {\em canonical representation} of a continued fraction  (see \cite{Khinchin}, Section 1.2). Given the equation $[c_1,...,c_n]=p_n/q_n$, the choice of integers $p_n,q_n$ is of course not unique. However, we define a canonical choice of $p_n$ and $q_n$ for each $[c_1,...,c_n]$, by induction on $n$, as follows: If $n=1$ then set $p_1=c_1$ and $q_1=1$. Suppose the canonical representations of all continued fractions of length $n-1$ have been defined, then we declare $p_n$ and $q_n$, the canonical representation of $[c_1,...,c_n]$,  as given by 
$$p_n = c_1p_{n-1} - q_{n-1} \quad \quad \text{ and } \quad \quad q_n = p_{n-1},$$
where $p_{n-1}$ and $q_{n-1}$ are the canonical representatives of $[c_2,...,c_n]$. To make our definition plausible, note that 
$$
[c_1,...,c_n]  = c_1 - \frac{1}{[c_2,...,c_n]} = c_1 - \frac{1}{\frac{p_{n-1}}{q_{n-1}}} =\frac{ c_1 p_{n-1} - q_{n-1}}{p_{n-1}} = \frac{p_n}{q_n}
$$
From hereon out, whenever we write $[c_1,...,c_n]=p/q$, we shall take $p$ and $q$ to be the canonical representatives of $[c_1,...,c_n]$ without explicit say. 
\begin{remark} \label{odd-n}
Since $[c_1,...,c_n] = [c_1,...,c_n\pm 1, \pm 1]$, we can always assume, without loss of generality, that a knot $K_{p/q}$ equals $K_{[c_1,...,c_n]}$ with $n$ odd.
\end{remark}

The Goeritz form $G=G(D)$ associated to a particular projection $D$ of the knot $K$, is a symmetric, bilinear, non-degenerate form $G:\mathbb Z^N\times \mathbb Z^N\to \mathbb Z$ where $N$ depends on the diagram $D$ (for the benefit of the reader, we recount the definition of the Goeritz form $G$ in Section \ref{section-goeritz-mu}). By means of choosing a basis for $\mathbb Z^N$, we will allow ourselves to view $G$ as an $N\times N$ symmetric non-degenerate matrix, referred to as the Goeritz matrix. As such, it can be diagnalized over the rationals, i.e. one can find an $N\times N$ regular matrix $P$, with rational entries, such that $P^\tau G P$ is the diagonal matrix $Diag (a_1,...,a_N)$. We shall capture such a statement by writing 
$$P^\tau G P = \langle a_1 \rangle \oplus \langle a_2 \rangle \oplus ... \oplus \langle a_N\rangle$$
where $\langle a \rangle$ should be thought of as a matrix representative of a bilinear form on a $1$-dimensional rational vector space. With these mind, our main result is contained in the next theorem.
\begin{theorem} \label{main}
Let $K=K_{[c_1,...,c_n]}$ be the two bridge knot associated to the ordered collection $(c_1,....,c_n)$ of nonzero integers and assume that $n$ is odd (see Remark \ref{odd-n}). Let $G$ be the Goeritz matrix of $K$ associated to its projection $D_{[c_1,..,c_n]}$ as in Figure \ref{pic1}. 

Then there is a matrix $P\in Gl_N(\mathbb Q)$, where $N=|c_1|+|c_3|+...+|c_n|-1$, with $\det P=\pm 1$ and such that 
$$ P^\tau G P = \bigoplus_{i=1,3,5,...,n}  \left( \oplus _{k=1}^{|c_i|-1} \langle -\eps_i \textstyle \frac{k+1}{k} \rangle \right) \oplus  \bigoplus _{i=2,4,...,n-1} \left\langle \textstyle \frac{p_{i+1}}{c_{i+1} \,  p_{i-1}} \right\rangle. $$
Here $\eps_i = Sign(c_i)$ and $p_m$ is the numerator of the canonical representation of $[c_1,...,c_m] = \frac{p_m}{q_m}$, $m\le n$. Accordingly, the signature $\sigma (G)$ of the Goeritz matrix $G$ is given by
$$
\sigma (G)  =  \sum _{i=1,3,...,n} (\eps _i - c_i)  +\sum _{i=2,4,...,n-1} Sign\left(\textstyle \frac{p_{i+1}}{c_{i+1} \, p_{i-1}}\right)  
$$
\end{theorem}
It is well known that $\det K_{[c_1,...,c_n]}= |p_n|$, cf. \cite{BurdeZieschang}. Since the determinant of the matrix $P$ from Theorem \ref{main} is $\pm 1$, it follows that the determinant of $G$ agrees with the determinant of $P^\tau GP$. The latter can easily be seen to equal $\pm p_n$ thereby verifying this well know fact.

By the Gordon-Litherland formula \cite{GordonLitherland}, the signature $\sigma (K)$ of a knot $K$ can be computed as 
$$ \sigma (K) = \sigma (G) - \mu, $$
where $G$ is a Goeritz matrix of $K$ associated to a diagram $D$, and $\mu = \mu(D)$ is a \lq\lq correction term\rq\rq, also read off from $D$ (we provide a detailed description of $\mu$ in Section \ref{section-goeritz-mu}). Having computed the signature of $G$ for the case of a two bridge knot $K$ in Theorem \ref{main}, the Gordon-Litherland formula provides an impetus for computing $\mu$. While we are not able to do this in full generality, we provide explicit formulae for $\mu (D_{[c_1,...,c_n]})$ for $n=1,3,5$. More importantly, we prove the existence of a continued fraction expansion $[c_1,...,c_n]$ for every two bridge knot such that $\mu$ for the corresponding diagram for $D_{[c_1,...,c_n]}$ is vanishing and therefore $\sigma (G) = \sigma (K_{[c_1,...,c_n]})$. The details follow. 

The next proposition states the values of $\mu (D_{[c_1,...,c_n]})$ for $n=1,3,5$ (the case of $n=1$ being stated for completeness). Whether or not $K_{[c_1,...,c_n]}$ is a knot or a link, depends on the parities of the coefficients $c_i$, and we only focus on those leading to knots. To reduce the number of parity choices to state, we note that the knots $K_{[c_1,...,c_n]}$ and $K_{[c_n,...,c_1]}$ are isotopic (the projection of one can be obtained from the other by two rotations by $180^\circ$). Thus their signatures are equal and so, for instance, rather than stating signature formulas for the parity cases $(\text{odd},\text{odd},\text{even})$ and $(\text{even},\text{odd},\text{odd})$, we only state one of these. 
\begin{proposition} \label{prop1}
We consider the diagram $D_{[c_1,...,c_n]}$ as in Figure \ref{pic1} and let $\mu = \mu(D_{[c_1,...,c_n]})$. The three tables below state the various parity conditions (up to symmetry) on the $c_i$ leading to knots and list the corresponding correction term $\mu$ and the signature of the knot. The quantities $\frac{p_3}{c_3 \, p_1}$ and $\frac{p_5}{c_5 \, p_3}$ appearing in the tables, can be computed from the coefficients $c_i$ explicitly as 
$$ \frac{p_3}{c_3 \,p_1} = c_2 - \frac{1}{c_1} - \frac{1}{c_3} \quad \quad \text{ and } \quad \quad \frac{p_5}{c_5 \, p_3} = c_4-\frac{1}{c_3}-\frac{1}{c_5} - \frac{1}{(c_3)^2(c_2-\frac{1}{c_1} - \frac{1}{c_3})}.$$
With this in mind, here are the tables:
%
$$
\begin{array}{c||c||c} 
\text{Parity of } c_1 & \mu & \sigma \left(K_{[c_1]}\right) \cr \hline \hline 
\text{odd} &0  & \eps_1 - c_1 
\end{array}
$$
\vskip3mm
%
$$
\begin{array}{c||c||c||c} 
No. & \text{Parity of }(c_1,c_2,c_3) & \mu & \sigma \left(K_{[c_1,c_2,c_3]}\right) \cr \hline \hline 
1& (\text{odd},\text{odd},\text{odd}) & -c_1+c_2-c_3 & \eps_1+\eps_3 -c_2+Sign(\frac{p_3}{c_3\, p_1})  \cr \hline
2 & (\text{odd},\text{odd},\text{even}) & -c_3 & \eps_1 -c_1+\eps_3 +Sign(\frac{p_3}{c_3\, p_1})   \cr \hline
3 & (\text{odd},\text{even},\text{even}) &  0 &\eps_1 -c_1+\eps_3 - c_3 +Sign(\frac{p_3}{c_3\, p_1})  
\end{array}
$$
\vskip3mm
$$
\begin{array}{c||c||c||c} 
No. & \text{Parity of }(c_1,c_2,c_3,c_4,c_5) & \mu & \sigma \left(K_{[c_1,c_2,c_3,c_4,c_5]}\right) \cr \hline \hline 
1 & (\text{odd},\text{odd},\text{odd},\text{odd}, \text{even}) & -c_1+c_2-c_3 & 
\begin{minipage}{45mm} $\eps_1-c_2 +\eps_3+ \eps_r - c_r +Sign(\frac{p_3}{c_3\, p_1})+Sign (\frac{p_5}{c_5\, p_3})$ \end{minipage} \cr \hline
2 & (\text{odd},\text{odd},\text{odd},\text{even}, \text{odd}) & -c_3+c_4-c_5 &  \begin{minipage}{45mm} $\eps_1-c_1+\eps_3-c_3-c_4+\eps_5+Sign(\frac{p_3}{c_3\, p_1})+Sign(\frac{p_5}{c_5\, p_3})$ \end{minipage}\cr \hline
3 & (\text{odd},\text{odd},\text{odd},\text{even}, \text{even}) & -c_1+c_2-c_3-c_5 &  \begin{minipage}{45mm} $\eps_1-c_2+\eps_3+\eps_5+Sign(\frac{p_3}{c_3\, p_1})+Sign(\frac{p_5}{c_5\, p_3})$ \end{minipage}\cr \hline
4 & (\text{odd},\text{odd},\text{even},\text{odd}, \text{even}) & -c_3 & \begin{minipage}{45mm} $\eps_1-c_1+\eps_3+\eps_5-c_5+Sign(\frac{p_3}{c_3\, p_1})+Sign(\frac{p_5}{c_5\, p_3})$ \end{minipage} \cr \hline
5 & (\text{even},\text{odd},\text{odd},\text{odd}, \text{even}) & -c_1-c_5 & \begin{minipage}{45mm} $\eps_1+\eps_3-c_3+\eps_5+Sign(\frac{p_3}{c_3\, p_1})+Sign(\frac{p_5}{c_5\, p_3})$ \end{minipage} \cr \hline
6 & (\text{odd},\text{odd},\text{even},\text{even}, \text{odd}) & \begin{minipage}{30mm} $-c_1+c_2-c_3+c_4-c_5$ \end{minipage} & \begin{minipage}{45mm} $\eps_1-c_2+\eps_3-c_4+\eps_5+Sign(\frac{p_3}{c_3\, p_1})+Sign(\frac{p_5}{c_5\, p_3})$ \end{minipage}  \cr \hline
7 & (\text{odd},\text{even},\text{odd},\text{even}, \text{odd}) & 0  & \begin{minipage}{45mm} $\eps_1-c_1+\eps_3-c_3+\eps_5-c_5+Sign(\frac{p_3}{c_3\, p_1})+Sign(\frac{p_5}{c_5\, p_3})$ \end{minipage} \cr \hline
8 & (\text{odd},\text{odd},\text{even},\text{even}, \text{even}) &  -c_3-c_5 & \begin{minipage}{45mm} $\eps_1-c_1+\eps_3+\eps_5+Sign(\frac{p_3}{c_3\, p_1})+Sign(\frac{p_5}{c_5\, p_3})$ \end{minipage} \cr \hline
9 & (\text{odd},\text{even},\text{even},\text{odd}, \text{even}) & -c_5   & \begin{minipage}{45mm} $\eps_1-c_1+\eps_3-c_3+\eps_5+Sign(\frac{p_3}{c_3\, p_1})+Sign(\frac{p_5}{c_5\, p_3})$ \end{minipage} \cr \hline
10 & (\text{even},\text{odd},\text{odd},\text{even}, \text{even}) & -c_1  & \begin{minipage}{45mm} $\eps_1+\eps_3-c_3+\eps_5-c_5+Sign(\frac{p_3}{c_3\, p_1})+Sign(\frac{p_5}{c_5\, p_3})$ \end{minipage} \cr \hline
11 & (\text{odd},\text{even},\text{even},\text{even}, \text{even}) & 0  & \begin{minipage}{45mm} $\eps_1-c_1+\eps_3-c_3+\eps_5-c_5+Sign(\frac{p_3}{c_3\, p_1})+Sign(\frac{p_5}{c_5\, p_3})$ \end{minipage} \cr \hline
12 & (\text{even},\text{even},\text{odd},\text{even}, \text{even}) &  0 & \begin{minipage}{45mm} $\eps_1-c_1+\eps_3-c_3+\eps_5-c_5+Sign(\frac{p_3}{c_3\, p_1})+Sign(\frac{p_5}{c_5\, p_3})$ \end{minipage}
\end{array}
$$
\end{proposition}
The next theorem furnishes an explicit computation of $\sigma (K_{p/q})$ provided one finds what we shall deem an {\em even continued fraction expansion of $p/q$}, i.e. a continued fraction expansion $p/q=[c_1,...,c_n]$ with $n$ odd and with $c_{2i}$ being even for all $2i\le n$. 
\begin{theorem} \label{main2}
Every two bridge knot is of the form $K_{[c_1,...,c_n]}$ with $[c_1,...,c_n]$ being an even continued fraction expansion. If $p/q=[c_1,...,c_n]$ is one such expansion, then the correction term $\mu$ of the associated diagram $D_{[c_1,...,c_n]}$ is zero and thus 
$$\sigma (K_{p/q}) = \sigma (G) = \sum _{i=1,3,...,n} (\eps_i - c_i)   + \sum _{i=2,4,...,n-1} Sign\left(  \frac{p_{i+1}}{c_{i+1} \, p_{i-1}}\right)$$ 
Here, as elsewhere, $p_i/q_i$ are the canonical representatives of $[c_1,...,c_i]$, $i\le n$.
\end{theorem}
\subsection{Applications and examples} 
The primary utility of Theorems \ref{main} and \ref{main2} is that of computing the signature $\sigma (K_{[c_1,...,c_n]})$. A formula for computing the signature of $K_{p/q}$ can be found in K. Murasugi's book \cite{Murasugi}, by which one forms the sequence $\{0,q,2q,3q,...,(p-1)q\}$, divides each entry by $2p$ and records its remainder $r$ with $-p<r<p$ to get a new sequence (of remainders) $\{0,r_1,r_2,...,r_{p-1}\}$. The signature $\sigma (K_{p/q})$ is then the number of positive entries in $\{0,r_1,...,r_{p-1}\}$ minus the number of negative entries\footnote{This algorithm for computing $\sigma (K_{p/q})$ assumes that $0<q<p$ and that $q$ is odd, both of which can always be achieved for any two bridge knot.}.  

The formula for $\sigma (K_{[c_1,...,c_n]})$ provided by Theorems \ref{main} and \ref{main2} is of a rather different nature, relying on the coefficients of the continued fraction expansion of $p/q$ rather than $p$ and $q$ themselves. This can lead to significantly shorter computations in some examples. For instance, considering the knots $K_{3023/151}$ and $K_{52587/4825}$
and using the continued fraction expansions 
$$ 3023/150=[20,-50,3] \quad \quad \text{ and } \quad \quad 52587/4825 =[11,10,9,8,7],$$
the formulae from Proposition \ref{prop1} readily yield (using line $3$ in the second table and line $7$ in the third table)
$$\sigma (K_{3023/151}) = -22 \quad \quad \text{ and } \quad \quad \sigma(K_{52587/4825})=-22.$$
On the other hand, the algorithm from \cite{Murasugi} requires us to form lists of remainders with $3023$ and $52587$ elements respectively, and count the number of positive versus negative entries. The advantage of our approach becomes even more prominent when $p$ is larger still. 
\vskip3mm
Recall that a knot $K$ is called {\em slice} if it is the boundary of a smoothly and properly embedded $2$-disk in the $4$-ball $D^4$. Two knots $K_1$ and $K_2$ are called {\em concordant} if $-K_1\#{K}_2$ is slice (where $-K_1$ is the reverse mirror of $K_1$). The notion of concordance is an equivalence relation and its equivalence classes, under the operation $\#$ of connect summing, form an Abelian group $\mathcal C$ called the {\em concordance group}. The concordance group $\mathcal C$ is a central object in low dimensional topology with relevance and applications to the theory of $3$-manifolds and smooth $4$-manifolds. Even so, it remains rather poorly understood (see \cite{Jabuka} for a survey of recent results), not even the possible types of torsion elements of $\mathcal C$ are known. 

While the subgroup of $\mathcal C$ generated by two bridge knots is not known, P. Lisca \cite{Lisca1, Lisca2} was able to obtain a complete list of slice two bridge knots as well as a complete list of slice knots among twofold sums $K_1\#K_2$ of two bridge knots. Beyond this, little is known about when a sum $K_1\#...\#K_n$ of two bridge knots is slice. Since slice knots have signature zero and $\sigma (K_1\# ... \#K_n) = \sigma (K_1) + ... + \sigma (K_n)$, Theorems \ref{main} and \ref{main2} provide a computable obstruction to sliceness of $K_1\#K_2\#...\#K_n$. Here are a few of many possible examples along these lines.   
\begin{example}
Consider the knots $K_1 = K_{35/16}$, $K_2=K_{283/34}$ and $K_3=K_{1193/145}$. Then neither of the knots 
$$(\pm K_1) \# (\pm K_2) \# (\pm K_3)$$
can be slice. Since $35/16 = [2,-5,3]$, $283/34=[8,-3,11]$ and $1193/145=[8,-4,3,2,-6]$, Proposition \ref{prop1} shows that 
$$\sigma (K_{35/16}) = -2 \quad \quad \quad \sigma (K_{283/34}) = -10 \quad \quad \quad \sigma (K_{1193/145}) = -4$$
from which the above claim follows. 
\end{example}  
\begin{example}
Let $K_1=K_{187/26}$, $K_2=K_{1451/131}$ and $K_3=K_{715/23}$. Then the knot $K_1\# (n\cdot K_2)\#(m\cdot K_3)$ cannot be slice for any choice of $m,n\in \mathbb Z$. Here $n\cdot K$ stands for the $n$-fold connected sum of $K$ with itself. 

By Proposition \ref{prop1}, it follows that 
$$
\begin{array}{lcl}
187/26=[7,-5,5] &\quad  \Longrightarrow \quad & \sigma (K_1) = 6 \cr
1451/131 = [11,-13,10] & \quad  \Longrightarrow \quad & \sigma (K_2) = -10 \cr
715/23 = [31,-12,-2] & \quad  \Longrightarrow \quad & \sigma (K_3) = -30
\end{array}
$$
Thus, the signature of $(n\cdot K_2)\#(m\cdot K_3)$ is divisible by $5$, for any choice of $m,n\in \mathbb Z$, while the signature of $K_1$ is not.
\end{example}
\subsection{Organization}
The remainder of this article is organized as follows. Section \ref{section-goeritz-mu} reviews the definitions of the Goeritz matrix $G$ and the correction term $\mu$ associated to a knot diagram $D$. Section \ref{section-proof-of-main-theorem} is devoted to the proof of Theorem \ref{main} while the final Section \ref{section-small-correction-terms} addresses the statements from Proposition \ref{prop1} and Theorem \ref{main2}. Section \ref{section-small-correction-terms} also provides an explicit algorithm for finding an even continued fraction expansion for any given two bridge knot. 
\section{The Goeritz matrix $G$ and the correction term $\mu$} \label{section-goeritz-mu}
This section elucidates the definitions of the Goeritz matrix $G=G(D)$ and the correction term $\mu=\mu(D)$, both associated to a projection $D$ of a knot $K$. Our exposition largely follows the introductory section  from \cite{GordonLitherland}. 
\vskip1mm
Let $K$ be an oriented knot and let $D$ be a projection of $K$. We color the regions of $D$ black and white, giving it a checkerboard pattern. Our convention is that the unbounded region of  $D$ receive a white coloring, see Figure \ref{pic2} for an example. 
\begin{figure}[htb!] 
\centering
\includegraphics[width=10cm]{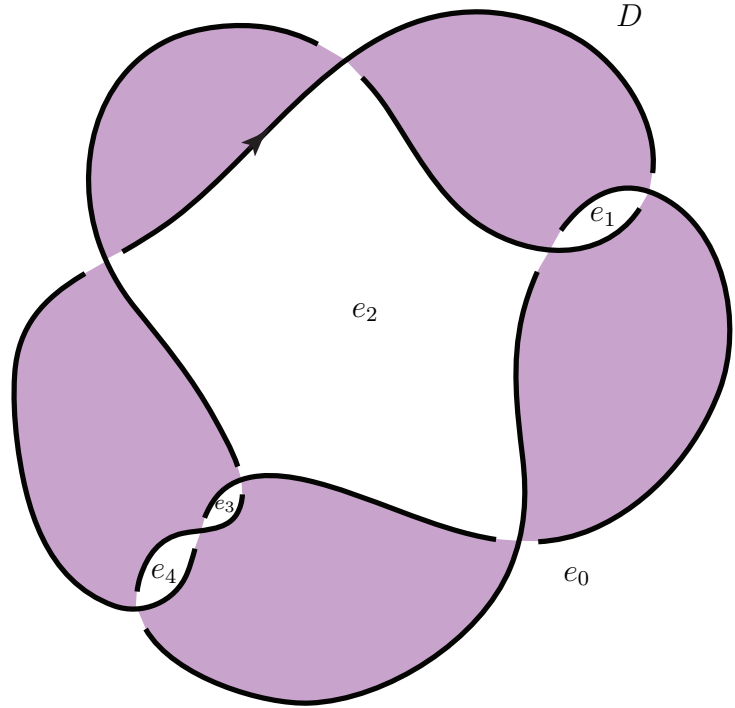}
\put(-50,270){$D$}
\put(-70,60){$e_0$}
\put(-60,196){$e_1$}
\put(-150,160){$e_2$}
\put(-202,87){\tiny $e_3$}
\put(-226,61){$e_4$}
\caption{The checkerboard black-and-white coloring of the regions of this diagram $D$ of the knot $K=8_6$ from the knot tables \cite{KnotInfo}.}  \label{pic2}
\end{figure}
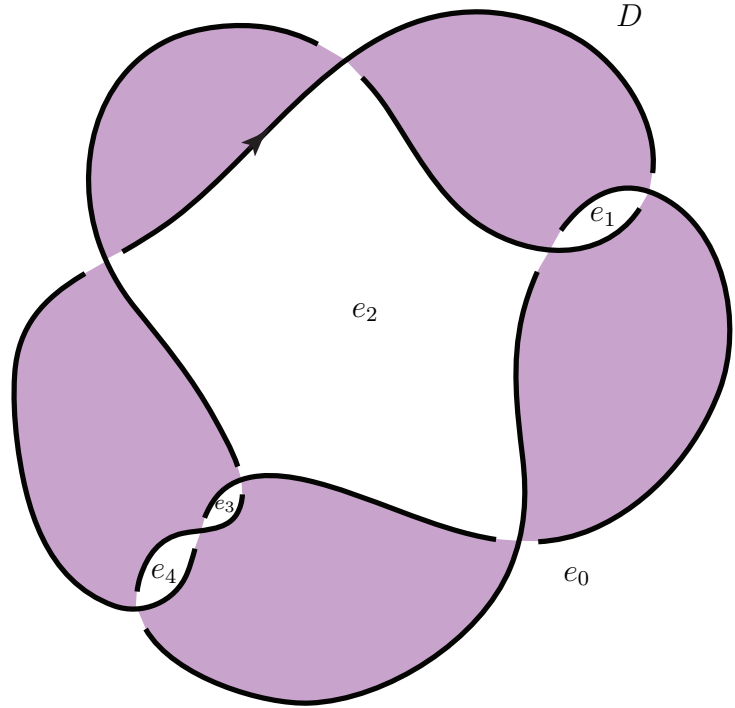

To each crossing $p_i$ in the diagram $D$, we associate two pieces of data, the sign of the crossing $\eta(p_i)$ and the type of the crossing $\tau(p_i)$. The computation of both is defined by Figure \ref{pic3}: 
\begin{figure}[htb!] 
\centering
\includegraphics[width=14cm]{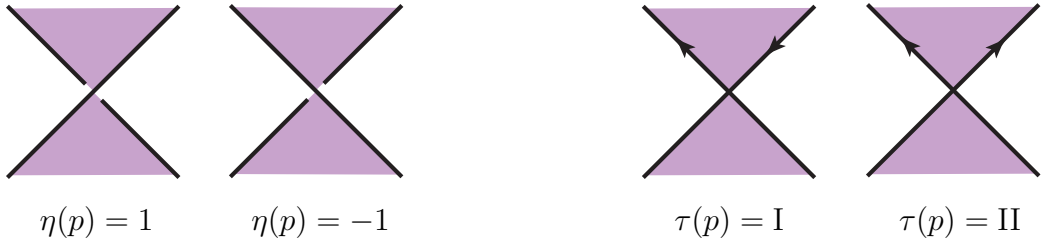}
\put(-382,-10){$\eta(p) = 1$}
\put(-302,-10){$\eta(p) = -1$}
\put(-142,-10){$\tau(p) = \text{I}$}
\put(-57,-10){$\tau(p) = \text{II}$}
\caption{The functions $\eta$ and $\tau$ assign to a double point $p$ the values $\pm 1$ and types I/II respectively. Note that $\eta(p)$ only depends on over/under-crossing information about $p$ while $\tau(p)$ only depends on the orientation of the diagram $D$ near $p$. }  \label{pic3}
\end{figure}
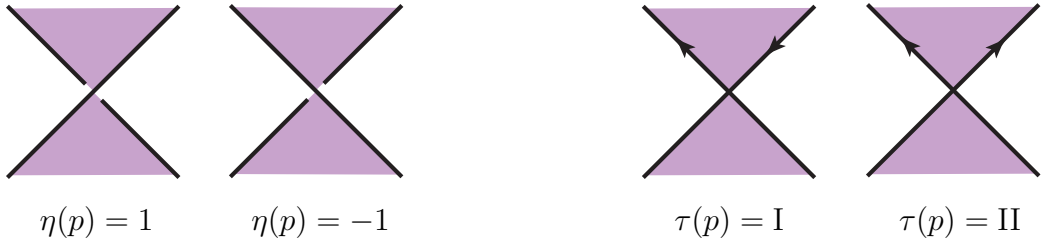

Let $\{e_0,...,e_N\}$ be a labeling of the white regions in the checkerboard pattern of $D$, with the convention that $e_0$ labels the unbounded region, and let $\mathbb Z^{N+1}$ be the free Abelian group generated by these symbols. Then the {\em pre-Goeritz form $PG$} is a bilinear symmetric form $PG:\mathbb Z^{N+1}\times \mathbb Z^{N+1} \to \mathbb Z$ whose associated matrix $[g_{ij}]$ with respect to the ordered basis $\{e_0,...,e_N\}$,  is given by 
$$g_{ij} = \left\{ 
\begin{array}{ll}
\displaystyle -\sum_{p\in e_i \cap e_j} \eta (p) & \quad ; \quad i\ne j  \cr
 & \cr
\displaystyle - \sum _{k\ne i} g_{ik} & \quad ; \quad i= j
\end{array}
\right. 
$$
The sum $\sum _{p\in e_i \cap e_j} \eta (p)$ is over all double points $p$ that connect the two white regions $e_i$ and $e_j$ in the diagram $D$. The {\em Goeritz form $G$} is obtained by restricting the form $PG$
to $\mathbb Z^{N}\times \mathbb Z^{N}$ where $\mathbb Z^{N}\subset \mathbb Z^{N+1}$ is obtained by discarding the $\mathbb Z$ summand generated by $e^0$. Since this construction relies on the choice of a basis of $\mathbb Z^N$, namely $\{e_1,...,e_N\}$,  we can, and often shall, think of $G$ as a symmetric $N\times N$ square matrix (with integer entries), called the {\em Goeritz matrix}. It follows from the work in \cite{GordonLitherland} that $\det G = \pm \det K$ so that $G$ is in fact a regular matrix. 

For simplicity of notation, we adopt the following convention which will substantially simplify our computations in the next section:  
\begin{equation} \label{convention} 
G(a,b) = \langle a, b \rangle \quad \quad \forall a,b\in \mathbb Z^{N}
\end{equation}

For example, the Goeritz matrix associated to the diagram $D$ and the basis $\{e_1,e_2,e_3,e_4\}$ from Figure \ref{pic2}, is given by 
$$G= \left[
\begin{array}{rrrr}
-2 & 1 & 0 & 0\cr
1 & -5 & 1 & 0 \cr
0 & 1 & -2 & 1 \cr
0 & 0 & 1 & -2
\end{array}
\right]$$ 
We leave it as an exercise to show that the signature of this matrix $G$ is $\sigma (G) = -4$ and its determinant is $\det G =23$.
 
The {\em correction term $\mu=\mu(D)$} associated to an oriented knot diagram $D$ is computed as 
$$ \mu = \sum _{\tau(p) = \text{II}} \eta (p)$$
In the above, the sum is taken over all double points $p$ of $D$ that are of type II. For example, for the diagram $D$ from Figure \ref{pic2}, one finds $\mu=-2$ (the only type II crossings are those adjacent to region $e_1$). 
With these understood, the following is proved in \cite{GordonLitherland}. 
\begin{theorem}[Gordon-Litherland \cite{GordonLitherland}]
Given any oriented diagram $D$ of a knot $K$, the signature $\sigma (K)$ of $K$ can be computed as 
$$ \sigma (K) = \sigma (G) - \mu.$$
Here $G$ and $\mu$ are the Goeritz matrix and the correction term associated to $D$.  
\end{theorem}
Using again the example from Figure \ref{pic2}, we can now compute the signature of $K=8_6$ by means of the above theorem:  
$$\sigma(8_6) = \sigma (G) - \mu = -4 -(-2) = -2.$$
Alternatively, the knot $8_6$ is the two bridge knot $K_{23/10}$ (according to KnotInfo \cite{KnotInfo}) and since $23/10=[2,-3,3]$, Proposition \ref{prop1} also shows that $\sigma (K_{23/10}) = -2$. 
\section{The proof of Theorem \ref{main}} \label{section-proof-of-main-theorem}
This section is devoted to the proof of Theorem \ref{main}. We first work out the Goeritz matrix associated to the specific diagram $D=D_{[c_1,...,c_n]}$ utilized in Figure \ref{pic1} and then proceed to diagonalize it by employing the Gram-Schmidt process. 
\vskip1mm
Let $n>0$ be an odd integer (compare Remark \ref{odd-n}), let $c_1,...,c_n$ be a collection of nonzero integers and let $K=K_{[c_1,...,c_n]}$ be the associated two bridge knot. Let $D=D_{[c_1,...,c_n]}$ be the diagram of $K$ as in Figure \ref{pic4} (see also Figure \ref{pic1}). 
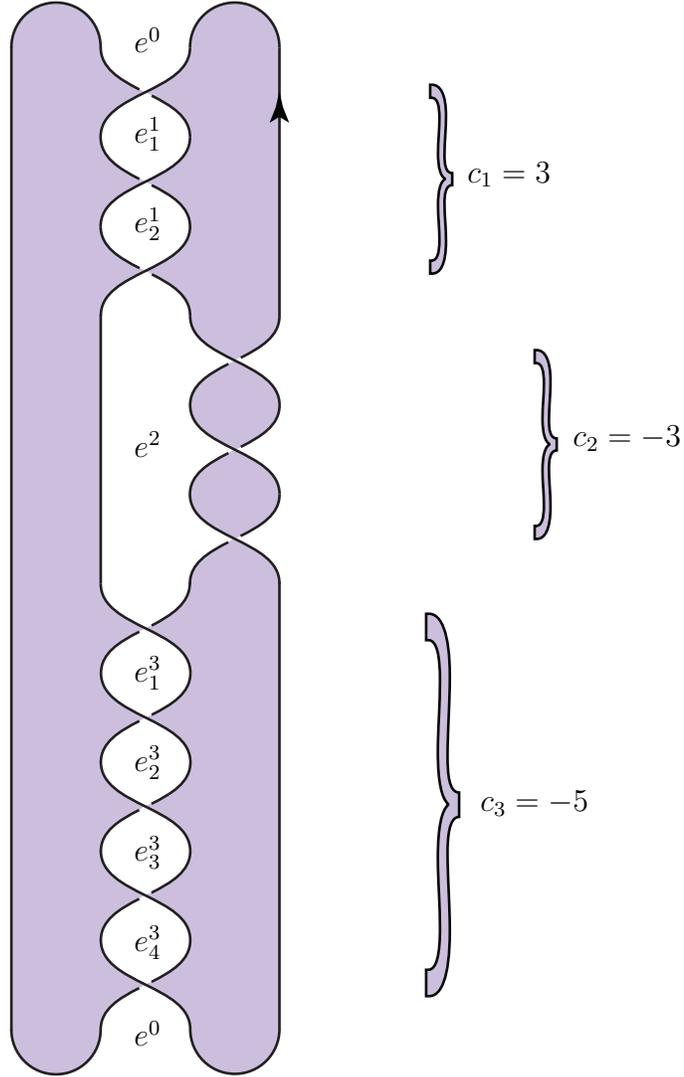
\begin{figure}[htb!] 
\centering
\includegraphics[width=10cm]{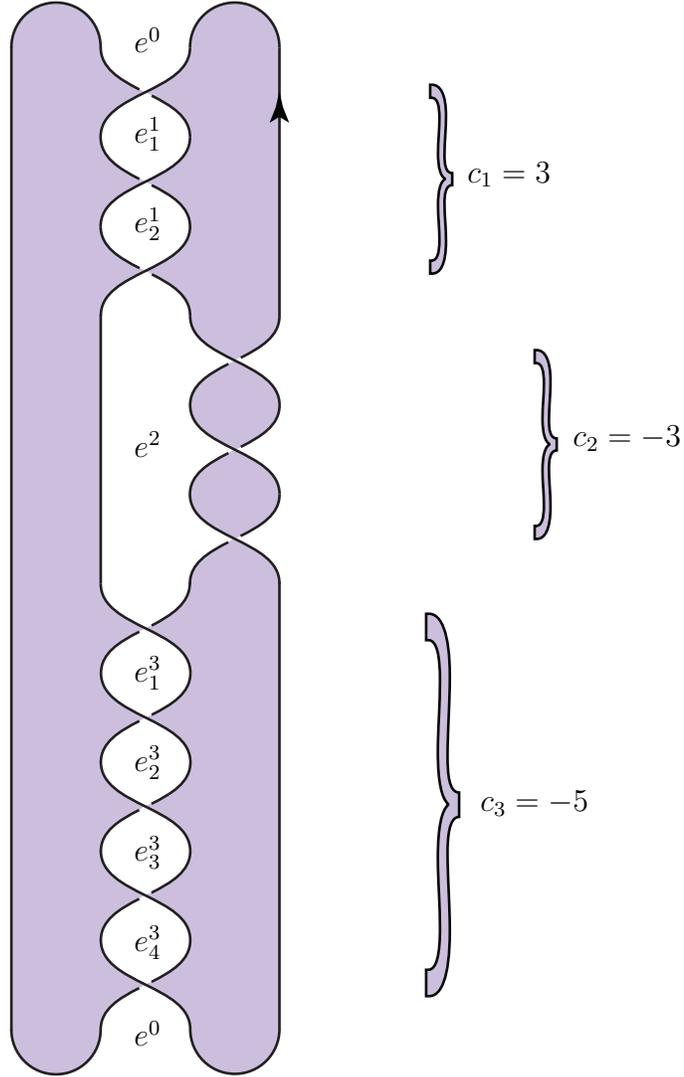}
\put(-196,398){$e^0$}
\put(-196,364){$e^1_1$}
\put(-196,330){$e^1_2$}
\put(-196,245){$e^2$}
\put(-196,160){$e^3_1$}
\put(-196,126){$e^3_2$}
\put(-196,92){$e^3_3$}
\put(-196,58){$e^3_4$}
\put(-196,22){$e^0$}
\put(-70,349){$c_1=3$}
\put(-30,249){$c_2=-3$}
\put(-65,111){$c_3=-5$}
\caption{The two bridge knot $K_{47/14}$ where $47/14=[3,-3,-5]$. }  \label{pic4}
\end{figure}
Give $D$ a checkerboard coloring and label its white regions as $e^0, e^2, ..., e^{n-1}$ and $e^1_1,...,e^1_{|c_1|-1}, ..., e^n_1,...,e^n_{|c_n|-1}$, where the labels are chose as:
\begin{align} \nonumber
e^0 & = \text{Unbounded white region.} \cr
e^{2i} & = \text{White region adjacent to the $c_{2i}$ half-twists.} \cr
e^{2i+1}_1,..., e^{2i+1}_{|c_{2i+1}|-1} & = \text{White regions adjacent to the $c_{2i+1}$ half-twists.}
\end{align}
These account for all white regions of $D$ showing that there is exactly 
$$N+1=|c_1|+|c_3|+...+|c_n| $$
of them (we express this number as $N+1$ since the region $e^0$ is discarded eventually when passing from the pre-Goertiz to the Goeritz matrix). The ordering of this basis for $\mathbb Z^N$ that we prefer to use is 
$$\{e^0, e^1_1,...,e^1_{|c_1|-1}, e^3_1,...,e^3_{|c_3|-1}, ..., e^n_1,...,e^n_{|c_n|-1}, e^2, e^4, ...., e^{n-1} \}$$

Recall our convention \eqref{convention} by which we write $\langle a,b \rangle$ for $G(a,b)$. With this in mind, it is easy to see, by consulting Figure \ref{pic4}, that 
\begin{align}\label{relations-for-the-es}
\langle e^i_k, e^j_\ell \rangle & = \left\{
\begin{array}{cl}
-2\eps_i & \quad ; \quad i=j \text{ and } |k-\ell|=0 \cr
\eps _i   & \quad ; \quad i=j \text{ and } |k-\ell|=1 \cr
0           & \quad ; \quad i\ne j \text{ or } |k-\ell |\ge 2    
\end{array}
\right. \cr
& \cr
\langle e^i, e^j_\ell\rangle & =  \left\{
\begin{array}{cl}
-\eps_{i-1} & \quad ; \quad j=i-1 \text{ and } \ell=|c_{i-1}-1 \cr
-\eps _{i+1}   & \quad ; \quad j=i+1 \text{ and } \ell=1 \cr
0           & \quad ; \quad \text{otherwise }   
\end{array}
\right. \cr
& \cr
\langle e^i , e^j \rangle & =  \left\{
\begin{array}{cl}
c_i-\eps_{i-1}-\eps_{i+1} & \quad ; \quad i=j \cr
0           & \quad ; \quad i\ne j    
\end{array}
\right. 
 \end{align}
In the above, we have used the abbreviation 
$$\eps _ i = Sign(c_i) $$
which we shall retain for the remainder of this section. The above relations capture the Goeritz matrix $G$ associated to the diagram $D=D_{[c_1,...,c_n]}$: 
\begin{equation} \label{matrix-of-G}
 \tiny G= 
 \left[
\begin{array}{ccccc|c|cccc|c|c|c|c}
&&&&   &\dots&& && &&0&\dots&0 \cr
&&&&   &\dots&& &&&&0&\dots&0 \cr
&&\eps_1 X_{|c_1|-1}&  &&\dots&& &0& &&\vdots&\dots&\vdots\cr 
&&&&  &\dots&& &&&&0&\dots&0 \cr
&&&&  &\dots&& &&&&\eps_1&\dots&0 \cr \hline 
\vdots&\vdots&\vdots&\vdots  &\vdots&\ddots&\vdots&\vdots &\vdots&\vdots&\vdots&\vdots&\ddots&\vdots\cr \hline
&&&&   &\dots&& &&&&0&\dots&\eps _n \cr
&&&&  &\dots&& &&&&0&\dots&0 \cr
&&0&&  &\dots&&&\eps_n X_{|c_n|-1} &&&\vdots&\dots&\vdots \cr 
&&&&  &\dots&& && &&0&\dots&0\cr
&&&&  &\dots&& &&&&0&\dots&0 \cr \hline 
0&0&\dots&0&\eps_1  &\dots&0& 0&\dots &0&0&c_2-\eps_1-\eps_3&\dots&0 \cr \hline 
\vdots&\vdots&\dots&\vdots &\vdots&\dots&\vdots& \vdots&\dots&\vdots&\vdots&\dots&\ddots&\vdots \cr \hline 
0&0&\dots&0&0  &\dots&\eps_n& 0&\dots &0&0&0&\dots& c_{n-1}-\eps_{n-2}-\eps_n
\end{array}
\right],
\end{equation}
The symbol $X_m$ utilized in the description of $G$, denotes the $m\times m$ square matrix 
$$X_m = \left[ 
\begin{array}{rrrrrrr}
-2 & 1 & 0 & 0 &  \dots & 0 & 0  \cr 
1 & -2 & 1 & 0 &  \dots & 0 & 0  \cr 
0 & 1 & -2 & 1 &  \dots & 0 & 0  \cr
0 & 0 & 1 & -2 &  \dots & 0 & 0  \cr
\vdots & \vdots & \vdots & \vdots & \ddots & \vdots & \vdots  \cr
0 & 0 & 0 & 0 &  \dots & 1 & 0  \cr 
0 & 0 & 0 & 0 &  \dots & -2 & 1  \cr 
0 & 0 & 0 & 0 &  \dots & 1 & -2
\end{array}
\right]. $$
We now turn to the task of diagonalizing $G$. We do so by thinking of $G:\mathbb Z^N \times \mathbb Z^N\to \mathbb Z$ as a bilinear form, one whose matrix description \eqref{matrix-of-G} is a facet of having chosen the basis 
\begin{equation} \label{old-e-basis}
\mathcal E = \{ e^{1}_1, ..., e^1_{|c_1|-1}\, \, , \,  \, e^3_1,...,e^3_{|c_3|-1}  \, \, , ... , \, \, e^n_1,...,e^n_{|c_n|-1} \, \, , \, e^2, e^4, ..., e^{n-1}  \}
\end{equation}
for $\mathbb Z^N$. Our task then becomes to find a new basis for $\mathbb Z^N$, one with respect to which $G$ has a diagonal matrix representative. The new basis 
\begin{equation} \label{new-f-basis}
\mathcal F =  \{ f^{1}_1, ..., f^1_{|c_1|-1}\, \, , \,  \, f^3_1,...,f^3_{|c_3|-1}  \, \, , ... , \, \, f^n_1,...,f^n_{|c_n|-1} \, \, , \, f^2, f^4, ..., f^{n-1}  \}
\end{equation}
will be obtained in several steps, each of which follows the Gram-Schmidt procedure. We outline our steps through the four Lemmas \ref{lemma-auxx-1},  \ref{lemma-auxx-2},  \ref{lemma-auxx-3} and \ref{lemma-auxx-4}. By way of nomenclature, we shall say that $a,b\in \mathbb Z^N$ are {\em orthogonal} if $\langle a,b\rangle=0$ (i.e. if $G(a,b)=0$). A subset $A\subset \mathbb Z^N$ is orthogonal if $\langle a, b\rangle = 0$ for all $a,b \in A$. 
\begin{lemma} \label{lemma-auxx-1}
For any choice of $i\in \{1,3,5,...,n\}$ and $k\in \{1,...,|c_i|-1\}$, let $f^i_k$ be defined as
\begin{equation} \label{definition-of-fs-1}
f^i_k = \frac{\eps _i}{k} (e^i_1 + 2 e^i_2+3e^i_3+...+ke^i_k).  
\end{equation}
Then the set $\{f^i_k\}^{i=1,3,...,n} _{k=1,...,|c_i|-1}$ is orthogonal and 
$$\langle f^i_k,f^i_k \rangle = -\eps_i \, \frac{k+1}{k}.$$
\end{lemma}
\begin{proof}
Since $\langle e^i_k, e^j_\ell \rangle =0$ whenever $i\ne j$, it follows that $\langle f^i_k, f^j_\ell\rangle = 0$ for all $i\ne j$. When $i=j$, pick two indices $k< \ell$ from $\{1,...,|c_i|-1\}$, then 
\begin{align} \nonumber
k\ell \cdot \langle f^i_k, f^i_\ell \rangle & =  \left\langle \sum _{r=1}^k r e^i_r , \sum _{s=1}^\ell s e^i_s \right\rangle  \cr
& =\langle e^i_1,e^i_1 + 2e^i_2\rangle +  \sum _{r=2}^k (\langle re^i_r, (r-1)e^i_{r-1}+r e^i_r +(r+1)e^i_{r+1} \rangle \cr
& =\eps_i  \sum _{r=2}^k r(r-1)-2r^2+r(r+1) \cr
& = 0 
\end{align}
Similarly, computing $\langle f^i_k, f^i_k\rangle$ gives
\begin{align} \nonumber
k^2\cdot \langle f^i_k, f^i_k \rangle & = \langle \sum _{r=1}^k r e^i_r, \sum _{s=1}^k s e^i_s\rangle \cr
& =  \langle e^i_1, e^i_1 +2e^i_2\rangle +  \langle ke^i_k, (k-1)e^i_{k-1} +ke^i_k \rangle + \cr
& \quad \quad \quad \quad + \sum _{r=2}^{k-1} (\langle re^i_r, (r-1)e^i_{r-1}+r e^i_r +(r+1)e^i_{r+1} \rangle \cr
& = \eps _i( k(k-1) -2k^2) \cr
& = -\eps_i k (k+1)
\end{align}
as claimed. 
\end{proof}
Before proceeding, we remark that the relations \eqref{relations-for-the-es} and the defintion of $f^i_k$ \eqref{definition-of-fs-1}, imply the following 
\begin{equation} \label{equation-relation-among-ej-and-fik}
\langle f^i_k, e^j \rangle =  \left\{
\begin{array}{cl}
1 & \quad ; \quad j=i+1 \text{ and } k =|c_{i}|-1\cr & \cr
\frac{1}{k} & \quad ; \quad j=i-1 \text{ and } k =1,...,|c_{i}|-1 \cr & \cr 
0 & \quad ; \quad \text{otherwise}
\end{array}
\right. 
\end{equation}
We define the remaining elements $f^2,f^4,...,f^{n-1}\in \mathbb Z^N$ for the basis $\mathcal F$ from \eqref{new-f-basis} in two steps. The next lemma first defines elements $\hat f^2,...,\hat f^{n-1}$, each of which is orthogonal to the previously defined $f^i_k$ and with $\langle \hat f^i,\hat f^j\rangle =0$ whenever $|i-j|\ge 4$. These $\hat f^i$ shall then be further modified in Lemma \ref{lemma-auxx-3} to obtain the desired $f^i$. 

\begin{lemma} \label{lemma-auxx-2}
For $j=2,4,6,...,n-1$ we define $\hat f^j$ as 
$$ \hat f^j =e^j  + \frac{|c_{j-1}|-1}{c_{j-1}} \, f^{j-1}_{|c_{j-1}|-1}  + \sum _{k=1}^{|c_{j+1}|-1} \frac{\eps_{j+1}}{k+1} \, f^{j+1}_k. $$
Then each $\hat f^j$ is orthogonal to the set $\{f^i_k\}^{i=1,3,...,n} _{k=1,...,|c_i|-1}$, and additionally 
$$ \langle \hat f^i, \hat f^j \rangle = \left\{
\begin{array}{cl}
c_j - \textstyle \frac{1}{c_{j-1}} -\frac{1}{c_{j+1}} & \quad ; \quad j=i \cr
& \cr
\frac{1}{c_{i \pm 1}} & \quad ; \quad j=i\pm 2 \cr
& \cr
0 & \quad ; \quad \text{otherwise}  
\end{array}
\right.
$$
\end{lemma}
\begin{proof}
All of these are direct computations, some of which make implicit use of the formulas from \eqref{equation-relation-among-ej-and-fik}. To begin with, note that $\langle \hat f^j, f^i_k \rangle =0$ whenever $i\ne j\pm1$. For $i=j-1$ we similarly have that $\langle \hat f^j, f^{j-1}_k \rangle =0$ if $k\ne |c_{j-1}|-1$ while if $k=|c_{j-1}|-1$ then   
\begin{align} \nonumber
\langle \hat f^j, f^{j-1}_{|c_{j-1}|-1}  \rangle & = \left\langle e^j  + \frac{1}{c_{j-1}} \, f^{j-1}_{|c_{j-1}|-1}  + \sum _{k=1}^{|c_{j+1}|-1}  \frac{\eps_{j+1}}{k+1} \, f^{j+1}_k  , f^{j-1}_{|c_{j-1}|-1}  \right\rangle \cr
& = \langle e^j, f^{j-1}_{|c_{j-1}|-1} \rangle + \frac{1}{c_{j-1}} \langle f^{j-1}_{|c_{j-1}|-1}, f^{j-1}_{|c_{j-1}|-1} \rangle \cr
& = 1+  \left( \frac{|c_{j-1}|-1}{c_{j-1}} \cdot (-\eps_{j-1}) \frac{|c_{j-1}|}{|c_{j-1}|-1} \right)  \cr
& = 0
\end{align}
Turning to the same computation with $i=j+1$, we find
\begin{align} \nonumber
\langle \hat f^j, f^{j+1}_\ell \rangle & = \left\langle e^j  + \frac{|c_{j-1}|-1}{c_{j-1}} \, f^{j-1}_{|c_{j-1}|-1}  + \sum _{k=1}^{|c_{j+1}|-1}\frac{\eps_{j+1}}{k+1 } \, f^{j+1}_k  , f^{j+1}_\ell  \right\rangle \cr
& = \langle e^j, f^{j+1}_\ell\rangle + \sum _{k=1}^{|c_{j+1}|-1} \frac{\eps_{j+1}}{k+1 } \, \langle f^{j+1}_k, f^{j+1}_\ell \rangle \cr
& = \frac{1}{\ell} + \frac{\eps_{j+1}}{\ell +1} \langle f^{j+1}_\ell ,f^{j+1}_\ell  \rangle \cr  
& = \frac{1}{\ell}+  \frac{\eps_{j+1}}{\ell +1} \cdot (-\eps_{j+1}) \frac{\ell +1}{\ell}  \cr
& = 0
\end{align}
These last two calculations verify that $\hat f^j$ is orthogonal to $f^i_k$ for any choice of $i,k$. From the definition of $\hat f^j$, it follows that $\langle \hat f^j, \hat f^i\rangle =0$ whenever $|j-i|>2$. When $i=j-2$, the following computation proves one of the remaining claims of the lemma:
\begin{align} \nonumber
 \langle & \hat f^j,  \hat f^{j-2} \rangle  = \cr 
& \textstyle = \langle e^j  + \frac{|c_{j-1}|-1}{c_{j-1}} \, f^{j-1}_{|c_{j-1}|-1}  + \sum _{k=1}^{|c_{j+1}|-1}\frac{ \eps_{j+1}}{k+1} \, f^{j+1}_k  ,  \cr 
&  \hspace{75mm} \textstyle , e^{j-2}  + \frac{|c_{j-3}|-1}{c_{j-3}} \, f^{j-3}_{|c_{j-3}|-1}  +  \sum _{k=1}^{|c_{j-1}|-1}\frac{\eps_{j-1}}{k+1} \, f^{j-1}_k \rangle  \cr 
& \textstyle = \langle e^j, \frac{\eps_{j-1}}{|c_{j-1}|} f^{j-1}_{|c_{j-1}|-1} \rangle + \langle \frac{|c_{j-1}|-1}{c_{j-1}} f^{j-1}_{|c_{j-1}|-1} , e^{j-2} \rangle + \langle \frac{|c_{j-1}|-1}{c_{j-1}} f^{j-1}_{|c_{j-1}|-1} ,  \frac{\eps_{j-1}}{|c_{j-1}|} f^{j-1}_{|c_{j-1}|-1} \rangle   \cr
& \textstyle= \frac{1}{c_{j-1}} + \frac{1}{c_{j-1}} - \frac{1}{c_{j-1}}  \cr
&  \textstyle= \frac{1}{c_{j-1}}. 
\end{align}
The very last computation is that of $\langle \hat f^i, \hat f^i\rangle$,  to which we now turn. 
\begin{align} \nonumber
\langle \hat f^i, \hat f^i\rangle & = \langle e^i, e^i\rangle +\frac{(|c_{i-1}|-1)^2}{c_{i-1}^2} \langle f^{i-1}_{|c_{i-1}|-1},f^{i-1}_{|c_{i-1}|-1}\rangle + \sum _{k=1}^{|c_{i+1}|-1} \frac{1}{(k+1)^2}\langle f^{i+1}_k, f^{i+1}_k\rangle  \cr
& \quad \quad \quad \quad + 2\frac{|c_{i-1}|-1}{c_{i-1}} \langle e^i, f^{i-1}_{|c_{i-1}|-1}\rangle + 2\sum _{k=1}^{|c_{i+1}|-1} \frac{\eps_{i+1}}{k+1}\langle e^i, f^{i+1}_k\rangle \cr
& = \langle e^i, e^i\rangle -\frac{|c_{i-1}|-1}{c_{i-1}}  -\eps_{i+1} \sum _{k=1}^{|c_{i+1}|-1} \frac{1}{k(k+1)} \cr
& \quad \quad \quad \quad + 2\, \frac{|c_{i-1}|-1}{c_{i-1}}  +\eps_{i+1} \sum _{k=1}^{|c_{i+1}|-1} \frac{2}{k(k+1)} \cr
&=  \langle e^i, e^i\rangle +\eps_{i-1} \frac{|c_{i-1}|-1}{|c_{i-1}|}  +\eps_{i+1} \sum _{k=1}^{|c_{i+1}|-1} \frac{1}{k(k+1)} \cr
& = (c_i-\eps_{i-1}-\eps_{i+1}) +\eps_{i-1} \left( 1- \frac{1}{|c_{i-1}|}\right)  + \eps_{i+1}\left( 1 - \frac{1}{|c_{i+1}|}  \right)  \cr 
& = c^i -\frac{1}{c_{i-1}}-\frac{1}{c_{i+1}}
\end{align}
This completes the proof of the lemma. 
\end{proof}

\begin{lemma} \label{lemma-auxx-3}
Define the sequence $\lambda _ {2j} \in \mathbb Q$, $j\ge 1$ recursively as 
$$\lambda_2=c_2-\frac{1}{c_1} - \frac{1}{c_3} \quad \quad \text{ and } \quad \quad \lambda_{2j} = \left(c_{2j} -\frac{1}{c_{2j-1}}-\frac{1}{c_{2j+1}} \right) - \frac{1}{c_{2j-1}^2 \cdot \lambda _{2j-2}}, \quad j\ge 2.  $$
Using this sequence, we define the vectors $f^j$, for $j=2,4,...,n-1$, as 
$$f^2 = \hat f^2 \quad \quad \quad \text{ and } \quad \quad  \quad f^j = \hat f^j - \frac{1}{c_{j-1} \cdot \lambda _{j-2}} \, f^{j-2} \text{ for } j \ge 4.$$
Then the set $\{ f^i_k, f^j\}^{i=1,3,...,n; \, j=2,4,..,n-1}_{k=1,2,...,|c_i|-1}$ is orthogonal and $\langle f^j, f^j \rangle = \lambda _j$. 
\end{lemma}
\begin{proof}
The proof of this lemma, once again, is a straightforward  though tedious computation. It should be clear that $f^j \perp f^i_k$ for any choices of $i,j,k$. To verify the claim about $\langle f^j, f^i\rangle$, we proceed by induction on $j$ (and assume that $i\le j$). 

Since $f^2=\hat f^2$, it follows from Lemma \ref{lemma-auxx-2} that $\langle f^2, f^2 \rangle =\lambda_2$. Taking $j=4$, we obtain
\begin{align} \nonumber
\langle f^4, f^2 \rangle  & = \left\langle \hat f^4 - \frac{1}{c_3\cdot \lambda _2} f^2, \hat f^2 \right\rangle \cr
& = \langle \hat f^4, \hat f^2\rangle - \frac{1}{c_3 \cdot \lambda _2} \langle \hat f^2, \hat f^2\rangle  \cr
& = \frac{1}{c_3} - \frac{1}{c_3\cdot \lambda _2} \lambda _2 \cr
& = 0 \cr
& \cr
\langle f^4, f^4\rangle & =  \left\langle \hat f^4 - \frac{1}{c_3\cdot \lambda _2} \hat f^2,  \hat f^4 - \frac{1}{c_3\cdot \lambda _2} \hat f^2 \right\rangle \cr
& = \langle \hat f^4, \hat f^4\rangle - \frac{2}{c_3 \cdot \lambda _2} \langle \hat f^4, \hat f^2\rangle + \frac{1}{(c_3\cdot \lambda _2)^2} \langle \hat f^2, \hat f^2 \rangle \cr
& = \left( c_4-\frac{1}{c_3}-\frac{1}{c_5}\right) -  \frac{2}{(c_3)^2 \cdot \lambda _2}+ \frac{1}{(c_3\cdot \lambda _2)^2} \lambda _2 \cr
&=  \left( c_4-\frac{1}{c_3}-\frac{1}{c_5}\right) -  \frac{1}{(c_3)^2 \cdot \lambda _2} \cr
& = \lambda _4
\end{align}
For the step of induction, we suppose the lemma to have been proved for all $2i, 2j \le 2m-2$ and we turn to computing $\langle f^{2m}, f^{2i}\rangle$ (with $i\le m$). Firstly, suppose that $2i\le2m-4$:
\begin{align}\nonumber
\langle f^{2m}, f^{2i} \rangle & = \left\langle \hat f^{2m} - \frac{1}{c_{2m-1}\cdot \lambda _{2m-2}} f^{2m-2}, f^{2i}  \right\rangle = 0
\end{align}
Next, let's take $2i=2m-2$: 
\begin{align} \nonumber
\langle f^{2m}, f^{2m-2} \rangle & = \left\langle \hat f^{2m} - \frac{1}{c_{2m-1}\cdot \lambda _{2m-2}} f^{2m-2}, f^{2m-2} \right\rangle  \cr
& = \langle \hat f^{2m} , f^{2m-2}\rangle -\frac{1}{c_{2m-1}\cdot \lambda _{2m-2}} \langle f^{2m-2}, f^{2m-2}\rangle \cr
& = \left\langle \hat f^{2m}, \hat f^{2m-2} - \frac{1}{c_{2m-3}\cdot \lambda _{2m-4}} \hat f^{2m-4} \right\rangle -\frac{1}{c_{2m-1}} \cr
& = \frac{1}{c_{2m-1}} - \frac{1}{c_{2m-1}}  \cr
& = 0
\end{align}
It remains to address the case of $2i=2m$:
\begin{align} \nonumber
\langle f^{2m}, f^{2m}\rangle & =  \left\langle \hat f^{2m} - \frac{1}{c_{2m-1}\cdot \lambda _{2m-2}} f^{2m-2},\hat f^{2m} - \frac{1}{c_{2m-1}\cdot \lambda _{2m-2}} f^{2m-2}  \right\rangle \cr
& =  \langle \hat f^{2m} , \hat f^{2m} \rangle- \frac{2}{c_{2m-1}\cdot \lambda _{2m-2}} \langle f^{2m-2},\hat f^{2m} \rangle  + \frac{1}{(c_{2m-1}\cdot \lambda _{2m-2})^2} \langle f^{2m-2} ,  f^{2m-2}  \rangle \cr
& =  \left( c_{2m}-\frac{1}{c_{2m-1}}-\frac{1}{c_{2m+1}}\right) - \frac{2}{(c_{2m-1})^2\cdot \lambda _{2m-2}} +  \frac{\lambda _{2m-2}}{(c_{2m-1}\cdot \lambda _{2m-2})^2} \cr
& =  \left( c_{2m}-\frac{1}{c_{2m-1}}-\frac{1}{c_{2m+1}}\right) - \frac{1}{(c_{2m-1})^2\cdot \lambda _{2m-2}} \cr
& = \lambda _{2m} 
\end{align}
With this, the lemma is proved. 
\end{proof}
For the next lemma, the reader is asked to recall the definition of the {\em canonical representation} of the continued fraction $[c_1,...,c_n]$ by the rational number $p_n/q_n$ (discussed in the introduction of Section \ref{section-introduction}). 
\begin{lemma} \label{lemma-auxx-4}
Let $c_1,c_2,...,c_n$ (with $n$ odd) and $\lambda _i$, $i=2,4,...,n-1$ be as in Lemma \ref{lemma-auxx-3}. For $i=1,2,...,n$, let us introduce the relatively prime integers $p_i, q_i$ as the numerator and denominator of the canonical representation of $[c_1,...,c_i]$:
$$[c_1,c_2, ..., c_{i}] = \frac{p_i}{q_i}$$
Then $\lambda _ {2i} = \displaystyle \frac{p_{2i+1}}{c_{2i+1}\cdot p_{2i-1}}$ for each $i\in \{1,2,...,\frac{n-1}{2}\}$.
\end{lemma}
\begin{proof}
We start by noting the following recursive relations connecting the various $p_i$ and $q_i$ (see Theorem 1 in \cite{Khinchin}):
$$p_n=c_np_{n-1}-p_{n-2} \quad \quad \text{ and } \quad \quad q_n = c_nq_{n-1} - q_{n-2}$$
Let us set $\mu_{2i} = \frac{p_{2i+1}}{c_{2i+1}\cdot p_{2i-1}}$. To show that $\lambda _{2i} = \mu_{2i}$, it suffices to demonstrate that $\mu_{2i}$ satisfies the recursion relation 
$$\mu_2 = c_2-\frac{1}{c_1} - \frac{1}{c_3}\quad \text{ and } \quad \mu_{2i} = \left(c_{2i}-\frac{1}{c_{2i-1}}-\frac{1}{c_{2i+1}} \right) - \frac{1}{c_{2i-1}^2\cdot \mu_{2i-2}}$$
from Lemma \ref{lemma-auxx-3}. The first of these equations is evident (since $p_1=c_1$ and $p_3=c_1c_2c_3-c_1-c_3$). The second is established using the noted recursive relation for $p_i$:
\begin{align}\nonumber
\mu_{2i} & = \frac{p_{2i+1}}{c_{2i+1}\, p_{2i-1}} =  \frac{c_{2i+1}\, p_{2i}-p_{2i-1}}{c_{2i+1}\, p_{2i-1}} = \frac{p_{2i}}{p_{2i-1}} - \frac{1}{c_{2i+1}} \cr
& = \frac{c_{2i}\, p_{2i-1}-p_{2i-2}}{p_{2i-1}} -  \frac{1}{c_{2i+1}} \cr
& = c_{2i} -  \frac{1}{c_{2i+1}} - \frac{p_{2i-2}}{p_{2i-1}} \cr
& =  c_{2i} -  \frac{1}{c_{2i+1}} - \frac{c_{2i-1}\,  p_{2i-2}}{c_{2i-1} \,  p_{2i-1}} \cr
& = c_{2i} -  \frac{1}{c_{2i+1}} - \frac{ p_{2i-1}+p_{2i-3}}{c_{2i-1}  p_{2i-1}} \cr
& = c_{2i} -  \frac{1}{c_{2i+1}} -\frac{1}{c_{2i-1}} - \frac{1}{c_{2i-1}^2 \cdot \frac{p_{2i-1}}{c_{2i-1}\, p_{2i-3}}} \cr
& = c_{2i} -  \frac{1}{c_{2i+1}} -\frac{1}{c_{2i-1}} -\frac{1}{c_{2i-1}^2 \cdot \mu_{2i-2}}
\end{align}
This completes the proof of the lemma.
\end{proof}
Lemmas  \ref{lemma-auxx-1}--\ref{lemma-auxx-4} provide a proof of Theorem \ref{main}. Namely, since the basis 
$$\mathcal F =  \{ f^{1}_1, ..., f^1_{|c_1|-1}\, \, , \,  \, f^3_1,...,f^3_{|c_3|-1}  \, \, , ... , \, \, f^n_1,...,f^n_{|c_n|-1} \, \, , \, f^2, f^4, ..., f^{n-1}  \}
$$
is orthogonal with respect to the \lq\lq inner product\rq\rq $\langle \cdot , \cdot \rangle$ provided by $G$ (according to Lemma \ref{lemma-auxx-3}), it follows that the matrix representing $G$ with respect to the basis $\mathcal F$ is diagonal and its entries are $\{ \langle f^1_1,f^1_1\rangle, ..., \langle f^{n-1},f^{n-1}\rangle \}$. These latter quantities have been computed (Lemma \ref{lemma-auxx-1} and Lemmas \ref{lemma-auxx-3}, \ref{lemma-auxx-4}) and are 
$$\langle f^i_k, f^i_k \rangle = - \eps_i \, \frac{k+1}{k} \quad \quad \text{ and } \quad \quad \langle f^{2j}, f^{2j}\rangle = \frac{p_{2i+1}}{c_{2i+1} \, p_{2i-1}}$$
This proves Theorem \ref{main} after observing that the transition matrix $P$ from the old basis $\mathcal E$ \eqref{old-e-basis} to the new basis $\mathcal F$ \eqref{new-f-basis} of $\mathbb Z^N$, is upper triangular with $\pm 1$ entries on the diagonal. 
\section{Proof of Proposition \ref{prop1} and Theorem \ref{main2}} \label{section-small-correction-terms}
\subsection{Proof of Proposition \ref{prop1}} Proposition \ref{prop1} is a direct consequence of Theorem \ref{main} and a number of explicit computations of the correction term $\mu$. We shall only outline two cases from the tables from Proposition \ref{prop1}, the other cases follow similarly.

{\em Case of $K_{[c_1,c_2,c_3]}$ with all $c_i$ odd.} Consider the example of $K_{47/14}$ from Figure \ref{pic5}. As the figure shows, these choices of parities render all crossings to be of type II and so $\mu$ becomes $\mu = -c_1+c_2-c_3$, as in line $1$ of the second table from Proposition \ref{prop1}.  
\begin{figure}[htb!] 
\centering
\includegraphics[width=6cm]{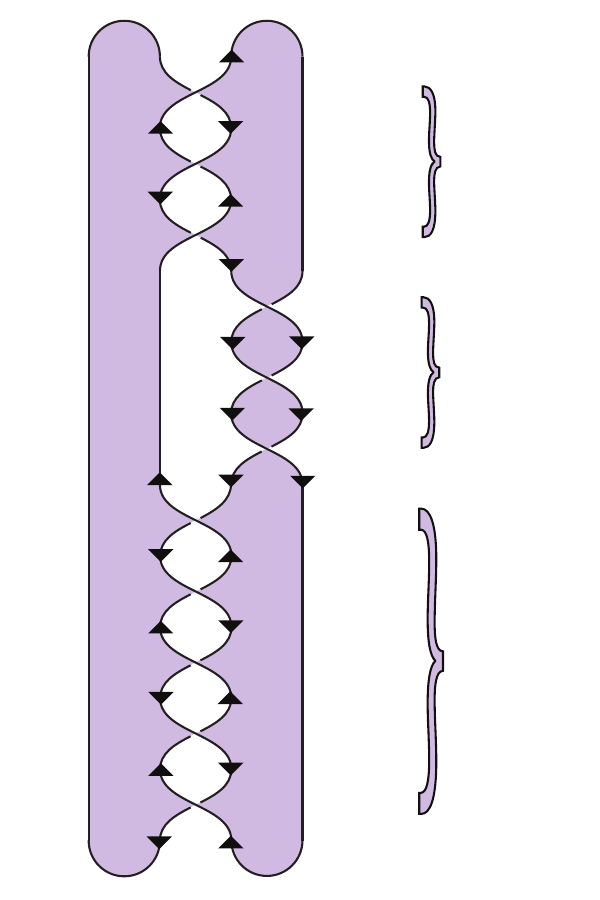}
\put(-50,208){\quad Type II crossings.}
\put(-50,148){\quad Type II crossings.}
\put(-50,66){\quad Type II crossings.}
\caption{The two bridge knot $K_{47/14}$ with $47/14=[3,-3,-5]$. All crossings are type II crossings so their contribution to the correction term $\mu$ is $-c_1+c_2-c_3$. }  \label{pic5}
\end{figure}
\vskip5mm
{\em Case of $K_{[c_1,c_2,c_3,c_4,c_5]}$ with $c_1, c_4$ odd and $c_2,c_3,c_5$ even.} Here we consider the knot $K_{11/3}$ with $11/3=[3,-2,-4,-1,-2]$, as in Figure \ref{pic6}. The figure shows that only the crossings stemming from $c_5$ are of type II and their contribution to $\mu$ is $-c_5$. This establishes the result in line 9 of the third table from Proposition \ref{prop1}.
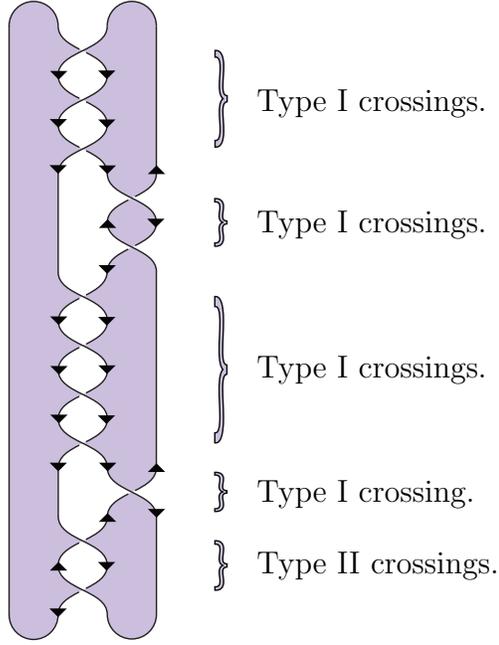
\begin{figure}[htb!] 
\centering
\includegraphics[width=4cm]{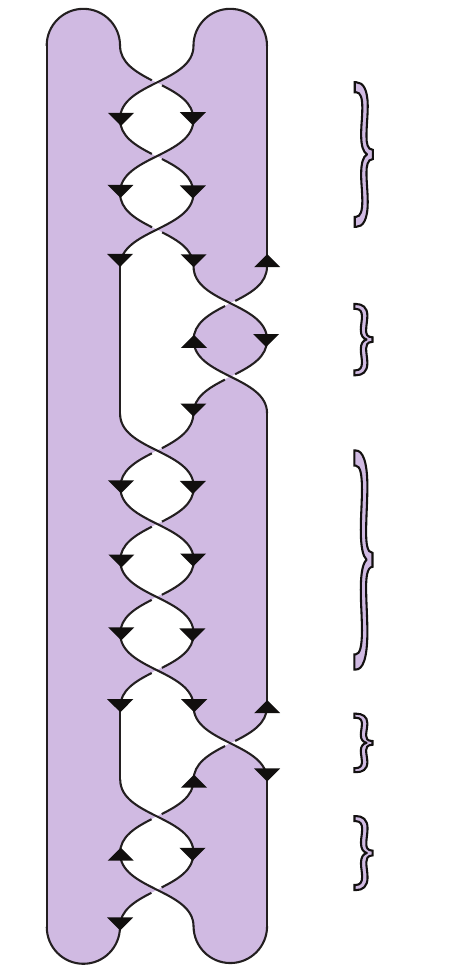}
\put(-20,203){\quad Type I crossings.}
\put(-20,157){\quad Type I crossings.}
\put(-20,102){\quad Type I crossings.}
\put(-20,55){\quad Type I crossing.}
\put(-20,28){\quad Type II crossings.}
\caption{The two bridge knot $K_{11/3}$ with $11/3=[3,-2,-4,-1,-2]$. The first four groups of crossing are of type I and the last one is of type II. Accordingly, $\mu$ equals $-c_5$.}  \label{pic6}
\end{figure}
\subsection{Proof of Theorem \ref{main2}} 
Consider a two bridge knot $K_{p/q}$. 
To find an even continued fraction expansion for $p/q$, we proceed with the following algorithm in $4$ steps. 
\vskip1mm
{\bf Step 1 } 
If $p/q$ is an integer, we simply let $c_1=p/q$ and note that $p/q=[c_1]$ is an even continued fraction expansion of $p/q$. 
\vskip1mm
{\bf Step 2 } 
Suppose that $p/q$ is not an integer and that $q>0$. Let $\varepsilon _{-1} = Sign(p)$ and replace $p$ by $|p|$ so that $p/q>0$. For uniformity of notation we introduce the abbreviations $r_{-1} = p $ and $r_0=q$. Write $\varepsilon _{-1} r_{-1}=c_1 r_0 \pm r_1$ with $0\le r_1<r_0$ so that  
$$ \frac{\varepsilon _{-1} r_{-1}}{r_0} = \frac{c_1 r_0 \pm r_1}{r_0} = c_1 \pm  \frac{r_1}{r_0} = c_1 - \frac{1}{\mp r_0/r_1}$$
The sign is chosen so as to make $c_1$ odd and we set $\varepsilon _0 = \mp1$ to keep track of our sign choice. 
\vskip1mm
{\bf Step 3 } 
This step inductively repeats Step 2 until the remainder $r_n$ becomes zero. Specifically, in the $n$-th step, having previously found $c_1,...,c_{n}$, $\varepsilon _{-1},...,\varepsilon _{n-1}$ and $r_{-1},...,r_{n-1}$, one writes $\varepsilon _{n-2} r_{n-2}$ as 
$$\varepsilon _{n-2} r_{n-2} = c_{n} r_{n-1} \pm r_{n}\quad \quad \text{ with } \quad \quad 0\le r_{n} < r_{n-1}$$
and with the sign $\varepsilon _n=\mp 1$ chosen so that $c_{n}$ is even (this can be done so long as $r_n\ne0$). 
Since the sequence $r_0,r_1, r_2,...$ is a strictly decreasing sequence of non-negative integers, this process eventually yields $r_n=0$ at which point we have produced a continued fraction expansion 
$$p/q = [c_1,...,c_n]$$
with $c_1$ odd and $c_{i}$ even for $i\ge 2$ with the possible exception of $c_n$. We note that $c_i\ne 0$ for all $i$ since $r_{i-2}<r_{i-1}$ for $i\ge 2$ and since $c_1$ was chosen to be odd. 
\vskip1mm
{\bf Step 4 } We consider the continued fraction expansion $[c_1,...,c_n]$ from Step 3.  If $n$ is odd, this continued fraction expansion is even and we are done. If $n$ is even and $c_n$ is odd, we change to the continued fraction expansion $[c_1,...,c_n\pm 1, \pm1]$ of $p/q$ which is even, and we are again done. Finally, if $n$ is even and $c_n$ is even, we consider the even continued fraction expansion 
$$ [1,1+c_1,c_2,...,c_n] = \frac{p}{p+q}$$
Since $K_{p/q}$ and $K_{p/(p+q)}$ are isotopic \cite{Murasugi}, we are done. 

\begin{example} 
We illustrate the above algorithm for $p/q = 137/37$. Note that Step 1 is omitted since $137/37$ is non-integral.

\noindent {\bf Step 2 } From $r_{-1}=137$, $\varepsilon_{-1}=1$ and $r_0=37$, we obtain $c_1=3$, $r_1=26$ and $\varepsilon _0=-1$. \\
{\bf Step 3 } Repeating the inductive Step $3$ four times, yields the table:
$$\begin{array}{c|c|c|c}
n & c_n & \varepsilon _{n-1} & r_n \cr \hline \hline 
2 & -2  & -1 & 15 \cr \hline  
3 & -2  & -1 & 4 \cr \hline
4 & -4  & -1 & 1 \cr \hline  
5 & -4  & 1 & 0
\end{array}
$$
From this one finds $137/37 = [3,-2,-2,-4,-4]$.\\
{\bf Step 4 } No further action is  required since the continued fraction $[3,-2,-2,-4,-4]$ is of odd length.
\end{example}

To finish the proof of Theorem \ref{main2}, we note that the black regions in the diagram $D_{[c_1,...,c_n]}$ (as in Figure \ref{pic4}) form an orientable Seifert surface for the knot $K_{[c_1,...,c_n]}$ whenever $[c_1,...,c_n]$ is an even continued fraction expansion. See Figure \ref{pic7} for an example. However, it is pointed out in \cite{GordonLitherland} that when this happens, the correction term $\mu$ vanishes so that the signature of the knot and its Goeritz matrix agree. This completes the proof of Theorem \ref{main2}.
\begin{figure}[htb!] 
\centering
\includegraphics[width=15cm]{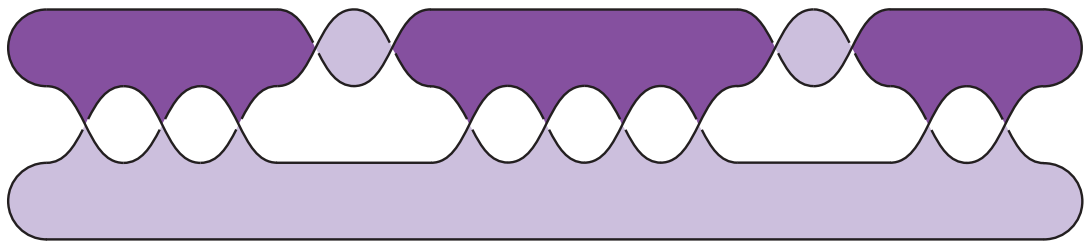}
\caption{This is the diagram $D_{[3,-2,-4,-2,-2]}$ of the two bridge knot $K_{61/17}$ associated to the even continued fraction expansion $61/17=[3,-2,-4,-2,-2]$. The oriented Seifert surface of $K_{61/17}$ formed by the black regions of the checkerboard pattern is clearly visible.}  \label{pic7}
\end{figure}
%
%
%

\end{document}